\documentclass[12pt]{amsart}
\usepackage{dynkin}

\title{Geometry of curves in generalized flag varieties}
\author
{Boris Doubrov \and Igor Zelenko}
\address{Belarusian State University, Nezavisimosti Ave.~4, Minsk 220030, Belarus;
 E-mail: doubrov@islc.org}
\address{Department of Mathematics, Texas A\&M University,
   College Station, TX 77843-3368, USA; E-mail: zelenko@math.tamu.edu}

\newcommand{\R}{\mathbb R}
\newcommand{\om}{\omega}
\newcommand{\g}{\mathfrak g}
\newcommand{\gl}{\mathfrak{gl}}
\newcommand{\pg}{\mathfrak p}
\newcommand{\sll}{\mathfrak{sl}}
\newcommand{\st}{\mathfrak{st}}
\newcommand{\sg}{\mathfrak s}
\newcommand{\hg}{\mathfrak h}
\newcommand{\Z}{\mathbb Z}
\newcommand{\Oct}{{\mathbb O}_s}
\newcommand{\ad}{\operatorname{ad}}
\newcommand{\gr}{\operatorname{gr}}
\newcommand{\Ad}{\operatorname{Ad}}
\newcommand{\Gr}{\operatorname{Gr}}
\newcommand{\tr}{\operatorname{tr}}
\newcommand{\diag}{\operatorname{diag}}
\newcommand{\Hom}{\operatorname{Hom}}
\newcommand{\RP}{\R\mathrm{P}}
\newcommand{\cG}{\mathcal G}
\newcommand{\xg}{\mathfrak{x}}
\newcommand{\eps}{\varepsilon}

\newcommand{\E}{\mathcal E}
\newcommand{\sol}{\operatorname{sol}}
\newenvironment{spmatrix}{\left(\begin{smallmatrix}}{\end{smallmatrix}\right)}

\newtheorem{thm}{Theorem}
\newtheorem{lem}{Lemma}

\theoremstyle{definition}
\newtheorem{dfn}{Definition}

\theoremstyle{remark}
\newtheorem{rem}{Remark}

\numberwithin{equation}{section}

\dedicatory{Dedicated to Mike Eastwood on the occasion of his 60th birthday.}

\begin{document}
\begin{abstract}
The current paper is devoted to the study of integral curves of constant type in generalized flag varieties. We construct a canonical moving frame bundle for such curves and give a criterion when it turns out to be a Cartan connection. Generalizations to parametrized curves, to higher-dimensional submanifolds and to parabolic geometries are discussed.
\end{abstract}

\maketitle

\section{Introduction}
 This paper is devoted to the local geometry of curves in generalized flag varieties. We provide the
uniform approach to this problem based on algebraic properties of so-called distinguished curves and their
symmetry algebras. In particular, we develop the methods that allow to deal with arbitrary curves without any
additional non-degeneracy assumptions. Although the paper deals with smooth curves in real homogeneous spaces, the whole theory applies equally to the case of holomorphic curves in complex generalized flag varieties.

Unlike the classical moving frame methods and techniques~\cite{cartan, griffits, green, jensen, olver1, olver2}, we avoid branchings as much as possible treating all curves at once. In fact, the only branching used in our paper is done at a very first step based on the type of the tangent line to a curve, or, in other words on its 1-st jet.

For each curve type there is a class of most symmetric curves, also known as \emph{distinguished curves} in a generalized flag variety~\cite{slovak, capslovak}. There are simple algorithms to compute the symmetry algebra of these curves \cite{dk}. We try to approximate any curve by the distinguished curves of the same type and construct the smallest possible moving frame bundle for all curves of a given type. It has the same dimension as the symmetry algebra of the distinguished curve, and its normalization conditions as well as the number of fundamental invariants (or curvatures) is defined by the simple algebraic properties of this symmetry algebra.

We are mainly interested in curves that are integral curves of the vector distribution defined naturally on any generalized flag variety. All distinguished curves of this kind admit the so-called~\emph projective parameter~\cite{dou:05}. This means that there is a natural 3-dimensional family of parameters on each such curve related by projective transformations. One of the questions we answer in this paper is when such projective parameter exists for any (non-homogeneous) curve of a given type. For example, it is well-known~\cite{cartan-proj} that an arbitrary non-degenerate curve in projective space can be naturally equipped with a family projective parameters, even if this curve is no longer homogeneous and, unlike rational normal curves (the distinguished curves in case of the projective space), is not invariant under the global action of $SL(2,\R)$.

It appears that the existence of such projective parameter in case of arbitrary generalized flag variety and any curve depends on some specific algebraic properties of the symmetry algebra of the distinguished curve
approximating a given curve. In particular, such parameter always exists when this symmetry algebra is reductive.

However, this is not always the case, and we give one of the smallest examples (curves in the flag variety $F_{2,4}(\R^5)$) when there are no invariant normalization conditions for the moving frame, and there
is no uniform way of constructing the projective parameter on all integral curves of generic type in $F_{2,4}(\R^5)$.

Let us introduce the notation we shall use throughout the paper. Let $M=G/P$ be an arbitrary generalized flag variety, where $\g=\sum_{i\in\Z}\g_i$ is a graded semisimple Lie algebra of the Lie group $G$ and $\pg=\sum_{i\ge 0}{\g}$ is a parabolic subalgebra of $\g$. It is well-known that there is a so-called \emph{grading element} $e\in\g_0$ such that $\g_i=\{u\in\g\mid [e,u]=iu\}$ for all $i\in\Z$. Denote by $G_0$ the stabilizer of $e$ with respect of the adjoint action of $G$ on $\g$. Then $G_0$ is a reductive, but not necessarily connected subgroup of $G$ with the Lie algebra $\g_0$.

Each generalized flag variety is naturally equipped with a bracket-generating vector distribution $D$, which can be described as follows. Let $o=eP$ be the ``origin'' in $M$. We can naturally identify $T_oM$ with a quotient space $\g/\pg$. Then the vector distribution $D\subset TM$ is defined as a $G$-invariant distribution equal to $\g_{-1}\mod \pg$ at $o$.

We are interested in local differential invariants of unparametrized integral curves of $D$. The first natural invariant of such curve $\gamma$ is a type of its tangent line. Namely, let $PD$ be the projectivization of the distribution $D$. The action of $G$ is naturally lifted to $PD$ and is, in general, no longer transitive. The orbits of this action are in one-to-one correspondence with the orbits of the action of $P$ (or, in fact, of $G_0$) on the projective space $P(\g_{-1})$. Each integral curve $\gamma$ is naturally lifted to the curve in $PD$. We say that $\gamma$ is \emph{of constant type}, if its lift to $PD$ lies in a single orbit of the action of $G$ on $PD$.

As proved by Vinberg~\cite{vinberg}, the action of the group $G_0$ on $P(\g_{-1})$ has a finite number of orbits. So, the above definition of a curve of constant type automatically holds for an open subset of $\gamma$ and essentially means that we exclude singular points, where the tangent lines to $\gamma$ degenerate to the boundary of the orbit of tangent lines at generic points.

The goal of this article is to provide a universal method for constructing \emph{moving frames} for integral curves of constant type on $M=P/G$. This study is motivated by a number of examples of curves in flag varieties,
which appear 
in the geometry of non-linear differential equations of finite type via the so-called linearization procedure ~\cite{dou:08, quasi}, non-holonomic vector distributions~\cite{douzel2,douzel3} and more general structures coming from the Control Theory \cite{agrzel2, agrzel3, zel-che, quasi} via the so-called symplectification/linearization procedure.

Note that our setup covers many classical works on the projective geometry of curves and ruled surfaces. Particular cases of curves in generalized flag varieties considered earlier in the literature also include the geometry of curves in Lagrangian Grassmannains~\cite{agrzel1,agrzel2,agrzel3,zel05,zel-che},  generic curves in $|1|$-graded flag varieties~\cite{beffa1,beffa2}, curves in Grassmann varieties $\Gr(r,kr)$ that correspond to systems of $r$ linear ODEs of order $k$~\cite{se-ashi}.

The paper is organized as follows. In Section~\ref{sec:2} we formulate the main result of the paper and prove it in Section~\ref{sec:3}. Section~\ref{sec:4} is devoted to various examples including the detailed description of all curve types for $G_2$ flag varieties.

In Section~\ref{sec:app} we apply the developed techniques for constructing conformal, almost-symplectic and $G_2$ structures on the solution space of a single ODE of order 3 and higher. These results can be considered a generalization of the classical Wunschmann result on the existence of a natural conformal structure on the solution space of a 3rd order ODE provided that a certain invariant of the equation (known today as Wunschmann invariant) vanishes identically. This results was generalized in~\cite{dou:08, dun-tod, godl, dun-godl, dou:dun} to the case of higher order ODEs which have all higher order Wunschmann invariants identically equal to $0$. In Theorem~\ref{thm:2} we find necessary and sufficient conditions for such geometric structures to exist.

Finally, in Section~\ref{sec:5} we discuss the generalizations of our main result to the cases of parametrized curves, curves in curved parabolic geometries and the case of higher dimensional integral submanifolds.

\subsection*{Acknowledgments}
The first author would like to thank the International Center for Theoretical Physics in Trieste and the Mathematical Sciences Institute at the Australian National University for their hospitality and financial support during the work on this paper. We also would like to thank Mike Eastwood, Tohru Morimoto, Maciej Dunajski,
Dennis The for many stimulating discussions on the topic.

\section{Formulation of the main result}\label{sec:2}
Let us fix an element $x\in \g$ of degree $-1$. We say that an integral curve $\gamma$ is \emph{of type $x$}, if the lift of $\gamma$ to $PD$ lies in the orbit of the line $\R x$ under the action of $G$ on $PD$. Here we treat the one-dimensional space $\R x$ as a point in $PD_o$. In other words, $\gamma$ is of type $x$ if for each point $p\in \gamma$ there exists an element $g\in G$ such that $g.o=p$ and $g_*(x)\subset T_p\gamma$.

Define a curve $\gamma_0\subset G/P$ as a closure of the trajectory of the one-parameter subgroup $\exp(tx)\subset G$ through the origin $o=eP$. It is known~\cite{dou:05} that we can complete $x$ to the basis $\{x,h,y\}$ of the subalgebra $\sll(2,\R)\subset \g$, such that $\deg h = 0$ and $\deg y = 1$. Then $\gamma_0$ is exactly the orbit of the corresponding subgroup in $G$ and, thus, is a rational curve or its finite covering. As we shall see, it is the most symmetric curve of type $x$. We call $\gamma_0$ \emph{a flat curve of type~$x$}.

Let $S$ be the symmetry group of $\gamma_0$, and let $\sg\subset \g$ be the corresponding subalgebra. It is known (see~\cite{dk}) that $\sg$ is a graded subalgebra of~$\g$, which can be also defined as follows:
\begin{equation}\label{symalg}
\begin{aligned}
\sg_{i} & = 0 \textrm{ for } i\le -2;\\
\sg_{-1} &= \langle x \rangle;\\
\sg_{i} &= (\ad x)^{-1}(\sg_{i-1}) = \{ u \in \g_i \mid [x,u]\in\sg_{i-1}\} \textrm{ for all } i\ge 0.
\end{aligned}
\end{equation}
Denote by $S^{(0)}$ the intersection of $S$ and $P$ and by $\sg^{(0)}$ the corresponding subalgebra of $\g$. It is clear that we have $\sg^{(0)}=\sum_{i\ge 0}\sg_i$.

Let us define two sequences of subgroups in $P$:
\begin{itemize}
\item $P^{(k)}$ is a subgroup of $P$ with the subalgebra $\pg^{(k)}=\sum_{i\ge k}\g_i$. It is equal to $P$ for $k=0$ and to $\prod_{i\ge k}\exp(\g_i)$ for $k\ge 1$. It is easy to see that $P^{(k)}$ is trivial for sufficiently large $k$.
\item $H^{(k)}$ is a subgroup of $P$ with the subalgebra $\hg^{(k)}=\sum_{i=0}^{k}\sg_i + \sum_{i>k}\g_i$. Denote by $S_0$ the intersection of $S$ and $G_0$. Then we have:
\[
H^{(k)} = S_0 \prod_{i=1}^k \exp(\sg_i) \prod_{i>k} \exp(\g_i), \quad \text{for any }k\ge 0.
\]
Lemma~\ref{lem:jet} below shows that $H^{(k)}$ is indeed a subgroup of $G$ and $\hg^{(k)}$ is a subalgebra of $\g$ for any $k\ge 0$. The subgroup $H^{(k)}$ becomes equal to $S^{(0)}=S\cap P$ for sufficiently large $k$. Note that $H^{(k)}$ contains $P^{(k+1)}$ for any $k\ge 0$.
\end{itemize}

The sequence of subgroups $H^{(k)}$ has a very natural geometric meaning:
\begin{lem}\label{lem:jet}
The subgroup $H^{(k)}$ is exactly the stabilizer of the $(k+1)$-th jet of the flat curve $\gamma_0$ of type $x$ at the origin $o=eP\in G/P$.
\end{lem}
\begin{proof}
First, assume $k=0$. The 1-jet of the curve $\gamma_0$ at the origin is determined by its tangent line. Under the natural identification of $T_o(G/P)$ with $\g/\pg$ this line is identified with the subspace $\langle x \rangle + \pg$. The stabilizer of the 1-jet coincides with the stabilizer of this subspace under the isotropic action of $P$ on $\g/\pg$. As $\deg x = -1$, it is clear that this stabilizer contains $P^{(1)}$. As $P=G_0P^{(1)}$ and $P^{(1)}$ is a normal subgroup in $P$, it remains to show that the intersection of the stabilizer with $G_0$ coincides with $S_0=S \cap G_0$.

It is clear that $S_0$ lies in the stabilizer of the 1-jet of $\gamma_0$ at the origin. Let now $g\in G_0$ and $\Ad g (x) \subset \langle x \rangle + \pg$. It implies that $\Ad g (x) =\alpha x$ for some non-zero constant $\alpha$, and, hence, $g\in S_0$.

The case $k\ge 1$ can be treated in the similar manner. We only need to note that the $(k+1)$-th jet of $\gamma_0$ is completely determined by the tangent line to the lift of $\gamma_0$ to the space of $k$-jets.
\end{proof}

Let $\pi\colon G\to G/P$ be the canonical principal $P$-bundle over the generalized flag variety $G/P$ and let $\om\colon TG\to \g$ be the left-invariant Maurer--Cartan form on $G$. By a \emph{moving frame bundle} $E$ along $\gamma$ we mean any subbundle (not necessarily principal) of the P-bundle $\pi^{-1}(\gamma)\to \gamma$.

Let us recall the basic notions of filtered and graded vector spaces as they will be used throughout the paper. Let $V=\sum_{i\in\Z} V_i$ be a (finite-dimensional) graded vector space. Then $V$ has a natural decreasing filtration $\{V^{(k)}\}$, where $V^{(k)}=\sum_{i\ge k}V_i$. If $U$ is any linear subspace in $V$, then we define $U^{(k)}=U\cap V^{(k)}$ and denote by $\gr U=\sum_{i\in\Z}U_i$ the corresponding graded subspace of $V$. Here $U_i=U^{(i)}/U^{(i+1)}$ is naturally identified with a subspace of $V_i$.

The role of ``normalization conditions'' for the canonical moving frame is played by a graded subspace $W\subset\pg$ complementary to $\sg^{(0)} + [x,\pg]$. In other words, let
$W=\sum_{i\ge 0} W_i$, where $W_i$ is a subspace in $\g_i$ complementary to $\sg_i + [x,\g_{i+1}]$.
\begin{dfn}\label{def1}
We say that a subspace $U\subset \g$ is \emph{$W$-normal (with respect to $\sg$),} if the following two conditions are satisfied:
\begin{enumerate}
\item $\gr U =\sg$;
\item $U$ contains an element $\bar x = \sum_{i\ge -1}x_i$, where $x_{-1}=x$ and $x_{i}\in W_i=W\cap\g_i$ for all $i\ge 0$.
\end{enumerate}
\end{dfn}
Note that the element $\bar x$ (and, hence, all $x_i$, $i\ge 0$) is defined uniquely for any $W$-normal subspace $U\subset\g$. Indeed, suppose there is another element $\bar x'=\sum_{i\ge -1}x_i'$ with this property. Then the element $\bar x - \bar x' = \sum_{i\ge 0} (x_i-x_i')$ also lies in $U$. As $\gr U=\sg$, we get $x_0-x_0'\in \sg_0$. On the other hand, both $x_0$ and $x_0'$ lie in $W_0$, which is complementary to $\sg_0+[x,\g_1]$ and, in particular, has zero intersection with $\sg_0$. Hence, $x_0=x_0'$. Proceeding by induction by $i$, we prove in the similar manner that $x_i=x_i'$ for all $i\ge 0$.

The main results of this paper can be formulated as follows.
\begin{thm}\label{thm1}
Fix any graded subspace $W\subset\pg$ complementary to $\sg^{(0)} + [x,\pg]$. Then for any curve $\gamma$ of type $x$ there is a unique frame bundle $E\subset G$ along $\gamma$ such that the subspace $\om(T_pE)$ is $W$-normal for each point $p\in E$ .

If, moreover, $W$ is $S^{(0)}$-invariant, then $E$ is a principal $S^{(0)}$-bundle, and $\om|_{E}$ decomposes as a sum of the $W$-valued 1-form $\om_W$ and the $\sg$-valued 1-form $\om_{\sg}$. The form $\om_\sg$ is a flat Cartan connection on $\gamma$ modeled by $S/S^{(0)}$, and $\om_W$ is a vertical $S^{(0)}$-equivariant 1-form defining a fundamental set of differential invariants of $\gamma$.
\end{thm}
The construction of thus moving frame bundle will be done in the next section.
\begin{rem}
The homogeneous space $S/S^{(0)}$ is exactly the model for the flat (or distinguished) curve of type $x$. As the it is (locally) isomorphic to $\RP^1$, the Cartan connection $\om_\sg$ defines the projective parameter on $\gamma$ in the canonical way.
\end{rem}
\begin{rem}
Here by \emph{a fundamental set of invariants} we mean that all the (relative or absolute) differential invariants of $\gamma$ can be obtained from the 1-form $\om_W$ and all its covariant derivatives defined by the connection $\om_\sg$. In particular, the curve $\gamma$ is locally flat if and only if $\om_W$ vanishes identically.
\end{rem}

Note that the $S^{(0)}$-invariant complement $W\subset\pg$ always exists, if the Lie algebra $\sg$ is reductive, or, what is equivalent, if the restriction of the Killing form of $\g$ to $\sg$ is non-degenerate.
Indeed, let $\sg^{\perp}$ be its orthogonal complement with respect to the Killing form of $\g$. As Killing form $K$ of $\g$ is compatible with the grading (i.e., $K(\g_i,\g_j)=0$ for all $i+j\ne 0$), the subspace $\sg^{\perp}$ is also graded: $\sg^{\perp} = \sum_{i\in \Z} \sg^{\perp}_i$, where $\sg^{\perp}_i = \sg^{\perp} \cap \g_i$.

Define the subspace $W\subset\pg$ as:
\begin{equation}\label{eq:w}
W = \{ u\in \sg^{\perp}\cap\pg \mid [u,y] = 0 \},
\end{equation}
where $y$ is an element of degree $+1$ in the $\sll_{2}$-subalgebra containing $x$.
\begin{lem}
The space $W$ is a $\sg^{(0)}$-invariant complement to $\sg^{(0)}+[x,\g_{+}]$ in $\pg$.
\end{lem}
\begin{proof}
First, note that $\sg^{\perp}$ is $\sg$-invariant and complementary to $\sg$. In particular, it is also stable with respect to the action of the $\sll_2$-subalgebra containing $x$. From the representation theory of $\sll_2$, it immediately follows that the space $W$ defined above is spanned by highest weight vectors (of non-negative degree) in the decomposition of $\sg^{\perp}$ into the direct sum of irreducible $\sll_2$-modules. This immediately implies that $W$ is complementary to $[x,\sg^{\perp}]$ in
$\sg^{\perp}\cap\pg$.

As the restriction of the Killing form $K$ to $\sg$ is non-degenerate and $\sg_{-1}=\R x$, we see that $\sg_{+1}=\R y$ and, in particular, $\R{y}$ is an ideal in $\sg^{(0)}=\sg_0+\sg_1$.  This proves that $W$ is $\sg^{(0)}$-invariant.
\end{proof}
\begin{rem}
The invariant complement $W$ may exist even if the Lie algebra $\sg$ is not reductive. See Subsection~\ref{g2-ex} for such example.
\end{rem}

\section{Canonical frame bundle}\label{sec:3}
\subsection{General reduction procedure}
Let $\gamma$ be an arbitrary unpara\-metrized curve in $G/P$ of type $\gamma_0$ and let $\pi\colon E\to \gamma$ be a moving frame bundle along the curve $\gamma$. For each point $p\in E$ the image $\om(T_pE)$ is a subspace in the Lie algebra $\g$. We say that $E$ is an $H$-bundle, where $H$ is a subgroup of $P$, if $E$ is preserved by the right action of $H$ on $G$. Note that we do not assume that $E$ is a principal $H$-bundles, i.e. the fibers of $\pi\colon E\to H$ are unions of orbits of $H$, but each fiber may in general be of larger dimension than $H$.

From definition of the Maurer--Cartan form we see that $\om(T_pE)$ contains $\hg$ for all $p\in E$, where $\hg$ is a subalgebra of $\g$ corresponding to the subgroup $H$. Thus, we can define a \emph{structure function} of the moving frame $H$-bundle $E$ as follows:
\[
c\colon E \to \Gr_r(\g/\hg),\quad p\mapsto \om(T_pE) /\hg\subset \g/\hg,
\]
where $r=\dim E - \dim H$ and $\Gr_r(\g/\hg)$ is a Grassmann variety of all $r$-dimensional subspaces in $\g/\hg$.

As the Maurer--Cartan form $\om$ satisfies the condition $R_g^*\om = \Ad g^{-1}\circ\om$ for any $g\in P$, we see that the structure function $c$ is $H$-equivariant:
\[
c(ph) = h^{-1}.c(p), \quad p\in E, h\in H,
\]
where the action of $H$ on $\Gr_r(\g/\hg)$ is induced from the adjoint action of $H$ on $\g$.

One of the main ideas of the reduction procedure is to impose certain \emph{normalization} conditions on the structure function $c$ and reduce $E$ to a subbundle $E'$ consisting of all points of $E$, such that $c(p)$ satisfies these imposed conditions. In general, $E'$ might not be a smooth submanifold of $E$, as $c(E)$ may not be transversal to the orbits of $H$ on $\Gr_r(\g/\hg)$. However, in our case we are always able to choose the normalization conditions in such a way that $E'$ is a regular submanifold in $E$ and, thus, is also a frame bundle along~$\gamma$.

\subsection{Zero order reduction}\label{ss:zero}
Let $\pi\colon G\to G/P$ be the canonical projection. Define the moving frame bundle $E_0$ as a set of all elements $g\in G$ satisfying the properties $g.o\in \gamma$ and $g_*(x)\in T_{g.o}\gamma$. Let $\pi_0$ be the restriction of $\pi$ to $E_0$. It is clear that $\pi_0\colon E_0\to \gamma$ is a principal $H^{(0)}$-bundle, as $H^{(0)}$ is exactly the subgroup of $P$ preserving the 1-jet of $\gamma_0$ at the origin.

Let us describe the structure function of the frame bundle $E_0$. By construction $\om(T_pE_0) \subset \langle x \rangle + \pg$ for all $p\in E_0$. As $r = \dim E_0 - \dim H^{(0)} = 1$, the structure function $c_0$ of the moving frame bundle $E_0$ takes values in the projectivization of the vector space $(\R x + \pg)/\hg^{(0)}\equiv (\R x + \g_0)/\sg_0$. In particular, for each $p\in E_0$ we can uniquely define an element $\chi_0(p)\in \g_0/\sg_0$ such that $x+\chi_0(p)\in c(p)=\om(T_pE_0)$. The structure function $c$ of $E_0$ is completely determined by this function $\chi_0\colon E_0 \to \g_0/\sg_0$.

\subsection{Further reductions and the canonical moving frame}
From now on we assume that we fixed a graded subspace $W\subset \pg$ complementary to $\sg^{(0)} + [x,\pg]$. Set $W_i=W\cap \g_i$ for any $i\ge 0$.

Let us introduce a slight generalization of the above notion of $W$-normality.

\begin{dfn}
We say that a subspace $U\subset \g$ is \emph{$W$-normal up to order $k$}, if it satisfies the following conditions:
\begin{enumerate}
\item $U \supset \pg^{(k)}$;
\item $\gr U = \sg+\pg^{(k)}$;
\item  $U$ contains an element $\bar x = \sum_{i=-1}^{k-1} x_i$, where $x_{-1}=x$ and $x_{i}\in W_i=W\cap\g_i$ for all $i=0,\dots,k-1$.
\end{enumerate}
\end{dfn}
Since $\g$ has a finite grading, $\pg^{(N)}=0$ for sufficiently large $N$ and the condition of $W$-normality up to order $N$ is equivalent to the notion of $W$-normality from the previous section.

We construct a canonical moving frame bundle for a given curve $\gamma$ via the sequence of reductions $E_k$ of the frame bundle $E_0$, where each $E_k$ is a moving frame $P^{(k+1)}$-bundle along $\gamma$.

For $k=0$ we can take as $E_0$ the zero-order frame bundle constructed above. Indeed, as $E_0$ is a principal $H^{(0)}$-bundle and $H^{(0)}$ contains $P^{(1)}$, it is invariant with respect to the right action of $P^{(1)}$ on $G$.

Assume by induction that for some $k\ge1$ we have constructed a moving frame $P^{(k)}$-bundle $E_{k-1}$, such that if $c_{k-1}$ is its structure function, then  $c_{k-1}(p)$ is $W$-normal up to order $k-1$ for any point $p\in E_{k-1}$.  Then we define $E_{k}$ as the following reduction of $E_{k-1}$:
\begin{equation}\label{reduct}
E_{k} = \{ p\in E_{k-1} \mid c_{k-1}(p) \textrm{ is $W$-normal up to order $k$} \}.
\end{equation}

\begin{lem}
\begin{enumerate}
\item The frame bundle $E_{k}$ is well-defined and is stable with respect to the
right action of the group $P^{(k+1)}$.
\item The structure function $c_k(p)=\om(T_pE_k)$, $p\in E_k$, satisfies the conditions: $c_k(p)\supset \pg^{(k+1)}$ and $\gr c_k(p) = \sg+\pg^{(k+1)}$ for any $p\in E_k$.
\item If, in addition, $W$ is $S^{(0)}$-invariant, then $E_k$ is stable with respect to the right action of the subgroup $H^{(k)}$.
\end{enumerate}
\end{lem}
\begin{proof}
Consider the structure function $c_{k-1}$ of the frame bundle $E_{k-1}$. By induction, the subspace $c_{k-1}(p)$ is $W$-normal up to order $k-1$. Thus, $c_{k-1}(p)$ contains an element:
\begin{equation}\label{eq:chi}
\chi(p) = x + \sum_{i=0}^{k-1}\chi_i(p),
\end{equation}
where $\chi_i(p)\in W_i$ are uniquely defined for $i=0,\dots,k-2$, and $\chi_{k-1}(p)$ is a certain element of $\g_{k-1}$ uniquely defined modulo $\sg_{k-1}$. Note that the condition~\eqref{reduct} is equivalent to:
\begin{equation}\label{reduct1}
\chi_{k-1}(p)\in W_{k-1}+\sg_{k-1}.
\end{equation}

Further, we see that:
\begin{align*}
\chi(pg) &= \chi(p) \mod \sg_{k-1}\quad\text{for any }g\in P^{(k+1)},\\
\chi(p\exp(u)) &= \chi(p)-[u,x] \mod \sg_{k-1} \quad\text{for any }u\in\g_k.
\end{align*}
The latter equation follows from the equivariance of Maurer--Cartan form:
\[
c_{k-1}(p\exp(u)) = \Ad (\exp(-u) )c_{k-1}(p),\quad \text{for any }p\in E_{k-1}, u\in\g_{k},
\]
 and the equality:
\[
\Ad \exp(-u) = \sum_{i=0}^{\infty} \frac{\ad^i(-u)}{i!},
\]
where this sum is actually finite for any element $u\in \g_k$, $k>0$.

As $P^{(k)}=\exp(\g_k)P^{(k+1)}$,  it follows that the condition~\eqref{reduct1} can always be achieved by an appropriate choice of the element $u\in\g_k$. Moreover, as $W_{k-1}$ is complementary to $\sg_{k-1}+[x,\g_k]$, the set of all elements $u\in\g_k$, which satisfy the condition $[x,u]\in\sg_{k-1}+W_{k-1}$ coincides with $\sg_k=\{u\in\g_k \mid [x,u]\in\sg_{k-1}\}$.
This proves that the reduction $E_k$ is well-defined and is stable with respect to $\exp(\sg_k)P^{(k+1)}$.

If $W$ is not $S^{(0)}$-invariant, we just reduce this group to $P^{(k+1)}$.  However, if $W$ is $S^{(0)}$-invariant, then the decomposition~\eqref{eq:chi} is stable with respect to the subgroup $H^{(k-1)}$, and, hence, the condition~\eqref{reduct1} defines the reduction of the principal $H^{(k-1)}$-bundle $E_{k-1}$ to the subgroup $H^{(k)}$.
\end{proof}

Thus, we see that $E_k$ is a moving $P^{(k+1)}$-bundle. If, in addition, the subspace $W$ is $S^{(0)}$-invariant, we treat $E_k$ as a principal $H^{(k)}$-bundle. As $\hg^{(k)}=\sg^{(0)}$ for sufficiently large $k$ and $\sg=\langle x \rangle + \sg^{(0)}$ we see that we finally get the moving frame bundle $E$ along $\gamma$ that satisfies conditions of Theorem~\ref{thm1}.

\subsection{The case of an invariant complementary subspace}
Suppose now that the subspace $W$ is in addition $S^{(0)}$-invariant. As it was shown above,
at each step $\pi\colon E_k\to \gamma$ is a principal $H^{(k-1)}$-bundle. For sufficiently large $k$ we get $H^{(k-1)}=S^{(0)}$ and $E=E_k$ is a principal $S^{(0)}$-bundle equipped with a 1-form $\om|_{E}$.
This implies that $\om(T_pE)$ contains $\sg^{(0)}$ for all $p\in E$. Hence, from $W$-normality condition we immediately get:
\[
\sg^{(0)}\subset \om(T_pE) \subset W\oplus\sg\quad\text{for all }p\in E.
\]
Let us decompose $\om|_{E}$ into the sum:
\[
\om|_{E} = \om_W + \om_{\sg},
\]
where the 1-forms $\om_W$ and  $\om_{\sg}$ take values in $W$ and $\sg$ respectively.

\begin{lem} The 1-form $\om_{\sg}$ is a flat Cartan connection with model $S/S^{(0)}$.
\end{lem}
\begin{proof}
From the properties of the Maurer--Cartan $\om$ it follows that $\om_{\sg}$ is $S^{(0)}$-equivariant and takes the fundamental vector fields on $E$ to the corresponding elements of the subalgebra $\sg^{(0)}$. By construction $\om_{\sg}(T_pE) + \sg^{(0)} = \sg$ for all $p\in E$, and, hence $\om_{\sg}$ is an isomorphism of $T_pE$ and $\sg$ for all $p\in E$. Thus, $\om_{\sg}$ is indeed a Cartan connection on $E$ modeled by the homogeneous space $S/S^{(0)}$.

As $\gamma$ is one-dimensional, this Cartan connection is necessarily flat.
\end{proof}

Consider now the 1-form $\omega_W$. It is clearly $S^{(0)}$-equivariant. It is also vertical, as $\om|_{E}$ takes all vertical vector fields to $\sg$. Finally, it is easy to see that $\gamma$ is locally equivalent to $\gamma_0$ (or, in other words, $\gamma$ is locally flat), if and only if $\om$ takes values in $\sg$, or, in other words, when $\om_W$ vanishes identically. Hence, $\om_W$ constitutes a fundamental system of invariants of $\gamma$.

\begin{rem}
Note that there are many more local differential invariants of $\gamma$. For example, we can construct
new invariants by taking total derivatives of $\om_W$ with respect to the Cartan connection $\om_{\sg}$.
\end{rem}

\section{Examples}\label{sec:4}
\subsection{Curves in projective spaces}\label{ssec:41}
As a very first example, consider a non-degenerate curve $\gamma\subset \RP^k$. Non-degeneracy means
that $\gamma$ does not lie in any proper linear subspace of $\RP^k$. Taking the osculating flag of $\gamma$,
we can also consider $\gamma$ as a curve in the complete flag manifold $F_{1,\dots,k-1}(k)$. We consider this flag
manifold as a generalized flag variety $G/P$, where $G=SL(k+1,\R)$ and $P$ is a Borel subgroup of $G$
consisting of all non-degenerate upper-triangular matrices. As a homogeneous curve $\gamma_0$ we
take the closure of the trajectory of $\exp(tx)$, where
\[
x = \begin{pmatrix} 0 & 0 & \dots & 0 & 0 \\ 1 & 0 & \dots & 0 & 0 \\ 0 & 1 & \ddots & 0 & 0 \\
0 & 0 & \ddots & 0 & 0 \\ 0 & 0 & \dots  & 1 & 0
\end{pmatrix}
\]
It is easy to  see that $\gamma_0$ is exactly the osculating variety of the rational normal curve in $\RP^k$.

It is clear that $\gamma$ can be approximated by $\gamma_0$ in the flag variety up to the first order. The symmetry algebra $\sg$ of $\gamma_0$ is isomorphic to $\sll(2,\R)$ embedded irreducibly into $\sll(k+1,\R)$. As $\sg$ is reductive, the restriction of the Killing form to $\sg$ is non-degenerate, and we can choose
$W$ as the set of all $y$-invariant elements of the orthogonal complement to $\sg$ with respect to the Killing form $K(A,B)=\operatorname{tr}(AB)$. Simple computation shows that $W = \langle y^i \mid i = 2, \dots, k\rangle$.

Thus, for any non-degenerate curve $\gamma$ in $\RP^k$ we can construct the canonical frame bundle $E\to \gamma$ such that $\om_\sg$ takes values in $\sll(2,\R)$ and defines a (flat) projective connection on it, while the 1-form $\om_W$ is equal to:
\begin{equation}\label{eq:omw}
\om_W = \sum_{i=2}^k \theta_{i+1}y^i.
\end{equation}
Here $\theta_i$ are the equivariant $1$-forms on the bundle $E$ defining a set of fundamental invariants of the curve $\gamma$.

These invariants can be constructed explicitly as follows. Let as choose an arbitrary local parameter $t$ on $\gamma$ and a curve $v_0(t)$ in $\R^{k+1}$ such that $\gamma(t) = \langle v_0(t) \rangle$ for all $t$. Define $v_i(t)=v_0^{(i)}(t)$. The osculating flag of $\gamma$ is given by subspaces $\gamma^{(i)} = \langle v_0(t), \dots, v_i(t) \rangle$ for all $0\le i \le k$. By regularity condition the family of vectors $\{ v_0(t), v_1(t),\dots, v_k(t) \}$ forms a basis in $\R^{k+1}$ for all $t$. In particular, this defines a section $s\colon \gamma \to GL(k+1,\R)$. Express the derivative of $v_k(t)$ in this basis:
\[
v_k'(t) = v_0^{(k+1)}(t) = p_0(t) v_0(t) + \dots + p_k(t) v_k(t).
\]
The pull-back of the Maurer-Cartan form $\om$ to the curve $\gamma$ has the form:
\begin{equation}\label{frame_pk}
s^*\omega =
\begin{pmatrix}
0 & 0 & \dots & 0 & p_0(t) \\
1 & 0 & \dots & 0 & p_1(t) \\
\vdots & \ddots & \vdots & \vdots \\
0 & 0 & \dots & 1 & p_k(t)
\end{pmatrix}.
\end{equation}
We see that the section $s$ takes values in $SL(k+1,\R)$ if and only if $\tr s^*\omega = p_k(t) = 0$. Multiplying the vector $v_0(t)$ by the function $\lambda(t)$, where $\lambda' = (k+1) p_k(t)\lambda$, we can always make the coefficient $p_k(t)$ vanish identically. Moreover, this defines the vector $v_0(t)$ uniquely up to the non-zero constant multiplier. Thus, for every choice of a local parameter $t$ we get a well-defined section $s\colon \gamma \to SL(k+1,\R)$. Moreover, from~\eqref{frame_pk} we see that these sections take values already in the zero-order reduction $E_0$ in terms of Subsection~\ref{ss:zero}.

In order to construct the section further reductions $E_i$, we can adjust the section $s$ by recalibrating it to $s\cdot A(t)$, where $A\colon \gamma \to S^{(0)}$ is any smooth function. It is easy to see that the connected component of the group $S^{(0)}$ has the form:
\[
S^{(0)} = \{\diag(a^k,a^{k-2},\dots,a^{-k})\mid a\in \R^*\} N(k+1,\R),
\]
where $N(k+1,\R)$ is the subgroup of all unipotent matrices in $SL(k+1,\R)$. The first multiplier is a subgroup in $S_0$, the symmetry group of the rational normal curve $\gamma_0$. So, we can consider only functions $h$ taking values in $N(k+1,\R)$. Let $A(t) = (A_{ij}(t))_{0\le i,j \le k}$, where $A_{ii} = 1$, $i=0,\dots, k$ and $h_{ij} = 0 $ for $i>j$. Then the condition that $sA(t)$ takes values in the canonical frame is equivalent to the condition that $A^{-1} (s^*\omega) A + A^{-1}A'$ is of the form:
\[
x + \sum_{i=2}^{k} \theta_{i+1}(t) y^i \mod \langle h, y\rangle.
\]

The normalization procedure of Section~\ref{sec:3} guarantees that such matrix $A(t)$ can always be found. Moreover, each reduction step involves only solving \emph{linear} equations for a part of the functions $A_{ij}$. In particular, the invariants $\theta_i(t)$ are polynomial functions of the coefficients $p_i(t)$, $i=0,\dots,k$ and their derivatives. For example, the invariant $\theta_3$ is explicitly given as:
\[
\theta_3 = p_{k}''-\frac6{k+1}p_kp_k'+\frac4{(k+1)^2}p_k^3+\frac{6}{k}p_{k-1}'
+\frac{12}{k(k+1)}p_kp_{k-1}+\frac{12}{k(k-1)}p_{k-2}.
\]

Invariants $\theta_i(t)$ were first constructed in the classical book by Wilczynski~\cite{wilch} in 1905 using the correspondence between the non-degenerate curves in $\RP^k$ and linear homogeneous differential equations of order $k+1$. See also~\cite{ovs-tab,se-ashi} for the definition of these invariants in the modern language. In the following we shall call them \emph{Wilczynski invariants} of the projective curve. As an easy consequence of our theory we immediately see that a curve $\gamma$ is an open part of a rational normal curve if and only if all its Wilczynski invariants vanish identically.

In the simplest case of curves in on the projective plane $P^2$ we recover this way the classical construction of \'Elie Cartan~\cite{cartan-proj} of the projective parameter and the projective invariants of plane curves considered up to projective transformations.

\subsection{Curves in Grassmann and flag varieties}
Let $\gamma$ be a curve in Grassmann variety $\Gr_r(V)$. In general, the tangent line $T_t\gamma$ for any $t\in\gamma$ can be identified with a one-dimensional subspace in $\Hom(\gamma(t),V/\gamma(t))$. Take any element $\phi_t$ in this line and denote by $\gamma'$ the subspace $\gamma(t) + \operatorname{Im}\phi_t$. We define $k$-th derivative $\gamma^{(k)}$ in a similar manner. Thus, for each $t\in \gamma$ we get a flag:
\[
  0\subset \gamma(t) \subset \gamma'(t) \subset \dots \subset V,
\]
which is called the osculating flag of $\gamma$. Assume that $\dim \gamma^{(i)}(t)=r_i$ for all $t\in\gamma$. Then the local geometry of $\gamma$ is equivalent to the local geometry of the corresponding curve of osculating flags. It is easy to see that this curve is an integral curve in the flag variety $F_{r_0,r_1,\dots,r_k}(V)$, the flag variety of the group $G=SL(V)$.

More generally, if
\[
0 \subset V_1(t)\subset V_2(t) \subset \dots \subset V_k(t)\subset V
\]
is any curve in the flag variety $F_{r_1,\dots,r_k}$, $r_i=\dim V_i(t)$, then it is an integral curve if and
only if it satisfies $V_i'(t)\subset V_{i+1}(t)$ for all $i=1,\dots,k-1$. Such curves appear naturally in the geometry of system of differential equations of mixed order (i.e. with in general non-constant highest order of derivatives for different unknown functions) via the linearization procedure \cite{quasi, mix1}.

Consider now the simplest example when the symmetry algebra $\sg$ of the flat curve is not reductive in $\g$. Let $\gamma$ be a curve in Grassmann variety $\Gr_2(\R^5)$. We assume that $\gamma$ satisfies two additional non-degeneracy conditions:
\[
\dim\gamma'(x)=4,\ \dim\gamma''(x)=5\quad\text{for all }x\in\gamma.
\]
These conditions mean that $\gamma$ can be lifted to the curve (we also denote it by $\gamma$) in the flag variety $F_{2,4}(\R^5)$, that can be approximated by the flat curve $\gamma_0$ corresponding to the following element $x\in \sll(5,\R)$:
\[
x = \begin{pmatrix} 0 & 0 & 0 & 0 & 0 \\ 0 & 0 & 0 & 0 & 0 \\ 1 & 0 & 0 & 0 & 0 \\
0 & 1 & 0 & 0 & 0 \\ 0 & 0 & 0 & 1 & 0 \end{pmatrix}.
\]
Direct computation shows that the subalgebra $\sg$ has the form:
\[
\left\{
\begin{spmatrix}
b & 0 & c & 0 & 0 \\ 0 & 2b & 0 & 2c & 0 \\ a & 0 & -b & 0 & 0 \\
0 & a & 0 & 0 & 2c \\ 0 & 0 & 0 & a & -2b
\end{spmatrix}
+
\begin{spmatrix} 3d & 0 & 0 & 0 & 0 \\ 0 & -2d & 0 & 0 & 0 \\ 0 & 0 & 3d & 0 & 0 \\
0 & 0 & 0 & -2d & 0 \\ 0 & 0 & 0 & 0 & -2d \end{spmatrix}
+
\begin{spmatrix} 0 & e_0 & 0 & e_1 & 0 \\ 0 & 0 & 0 & 0 & 0 \\ 0 & 0 & 0 & e_0 & 2e_1 \\
0 & 0 & 0 & 0 & 0 \\ 0 & 0 & 0 & 0 & 0 \end{spmatrix}\right\}.
\]
Here the first summand is an $\sll_2(\R)$-subalgebra containing an element $x$, the second summand is its centralizer and the third summand is the nilradical of $\sg$. Note that this is a particular case of Theorem 8.1  and Remark 8.2 of~\cite{flags} and it is based on the theory of $\sll_2$ representations.

Let us show that there is no $\sg^{(0)}$-invariant subspace $W\subset\pg$ complementary to $\sg^{(0)} + [x,\pg]$. Indeed, it is clear that such subspace $W$ should satisfy $W_2=\g_2$ and $\dim W_1=2$. Decomposing $\g_1$ into the direct sum of eigenspaces with respect to the action of the 2-dimensional diagonalizable subalgebra in $\sg_0$, we see that $W_1$ would necessarily be spanned by the following two elements:
\[
w_1 = \begin{spmatrix} 0 & 0 & 0 & 0 & 0 \\ 0 & 0 & 1 & 0 & 0 \\ 0 & 0 & 0 & 0 & 0 \\
0 & 0 & 0 & 0 & 0 \\ 0 & 0 & 0 & 0 & 0 \end{spmatrix},\quad
w_2 = \begin{spmatrix} 0 & 0 & \alpha_1 & 0 & 0 \\ 0 & 0 & 0 & \alpha_2 & 0 \\ 0 & 0 & 0 & 0 & 0 \\
0 & 0 & 0 & 0 & \alpha_3 \\ 0 & 0 & 0 & 0 & 0 \end{spmatrix}.
\]
However, the action of $\sg_0$ on these 2 elements generates at least 3-dimensional space in degree 1.

Thus, in general, we can not equip a non-degenerate curve in $\Gr_2(\R^5)$ with a natural Cartan connection as it was done for non-degenerate curves in projective spaces. However, by choosing a non-invariant complementary subspace $W$ we still can build a natural frame bundle along $\gamma$ satisfying the conditions of Theorem~\ref{thm1}.

\subsection{Curves of isotropic and coisotropic subspaces}
Let $V$ be a vector space equipped with a non-degenerate symmetric or antisymmetric form $b$. Denote by $\sll(V,b)$ the subalgebra in $\sll(V)$ preserving $b$. It is either $\mathfrak{sp}(V)$ or $\mathfrak{so}(V)$, but we prefer to use the unifying notation.

We shall denote by $\Gr_r(V,b)$ the variety of all $r$-dimensional isotropic subspaces in $V$. More generally, denote by $F_{r_1,r_2,\dots,r_k}(V,b)$ the variety of flags of isotropic subspaces $V_1\subset V_2 \subset \dots \subset V_k$ of dimensions $r_1,r_2,\dots, r_k$.

We shall consider curves $\gamma$ in $F_{r_1,r_2,\dots,r_k}(V,b)$ satisfying the following additional conditions:
\begin{align}
V_i' & \subset V_{i+1}, i =1,\dots,k-1;\\
V_k' & \subset V_{k-1}^\perp,\quad \text{if $V_k$ is Lagrangian};\\
V_k' & \subset V_{k}^\perp,\quad \text{if $V_k$ is not Lagrangian}.
\end{align}
Here by $V_i$ we mean the curve in $\Gr_{r_i}(V)$ defined by $\gamma$, and by $V_i'$ its derivative defined above.

In fact, these are exactly the types of curves that can be approximated up to first order by a flat curves $\gamma_0$ corresponding to elements $x$ of degree $-1$ of $\sll(V,b)$ equipped with an appropriate grading. In the case of $\mathfrak{sp}(V)$ such curves appear in the geometry of nonholonomic distributions \cite{douzel2,douzel3} and more general structures coming from the Control Theory \cite{agrzel2, agrzel3, zel-che, quasi}.

It is easy to check that $\sg$ is reductive if $r_i = is$ and $r_k = ks = \dim V/2$ for an appropriate $s$. In particular, we can find an $S^{(0)}$-invariant subspace $W$ in all these cases and equip the curve $\gamma$ with a Cartan connection (and a projective parameter) in a natural way.

In~\cite{flags} we provide the complete classification of all possible curve types and compute their symmetry algebras for all flag varieties of classical simple Lie algebras.

\subsection{$G_2$ examples}
\label{g2-ex}

Let $G$ be a split real Lie group of type $G_2$ and $\g$ be the corresponding Lie algebra. Up to conjugation there are three different connected parabolic subgroups of $G$, which can be easily described using the realization of $G$ as an automorphism group of split octonions. Let $\Oct$ be the algebra of split octonions as defined in~\cite[\S19.3]{humph}, and let $V$ be the 7-dimensional subspace of all imaginary octonions in $\Oct$. The multiplication of imaginary octonions split into scalar and imaginary parts defines the non-degenerate symmetric bilinear form $B$ on $V$ and a skew-symmetric bilinear map $\Omega\colon \wedge^2 V \to V$. Identifying $V$ with $V^*$ by means of $B$, we can identify $\Omega$ with an element of $\wedge^2 V^*\otimes V^*$, which turns to be totally antisymmetric. Thus, the multiplication of $\Oct$ defines a 3-form $\Omega\in \wedge^3 V^*$. In fact, it is also well-known that the form $\Omega$ defines $B$ uniquely up to a constant via the following identity:
\[
i_{v_1}\Omega \wedge i_{v_2}\Omega \wedge \Omega = B(v_1,v_2)\operatorname{vol},
\]
where $\operatorname{vol}$ is a volume form on $V$.

We identify the Lie group $G$ of type $G_2$ with the automorphism group of $\Oct$. It preserves $V$, the symmetric form $B$ and the 3-form $\Omega$ on $V$. We denote by $\Delta$ the root system of type $G_2$, which has the following Dynkin diagram:
\[
\cir[,\alpha_1]\lllar\cir[,\alpha_2]
\]
and the set of all positive roots:
\[
\alpha_1, \alpha_2, \alpha_1+\alpha_2, \alpha_1+2\alpha_2,\alpha_1+3\alpha_2,2\alpha_1+3\alpha_2.
\]
We denote the basis of $\g$ by $\{H_1,H_2,X_\alpha : \alpha\in\Delta\}$, where $\{H_1,H_2\}$ is the basis
of the Cartan subalgebra of $\g$ dual to the basis $\{\alpha_1,\alpha_2\}$ of the root system $\Delta$ and
$X_{\alpha}$ is any non-zero element in $\g_{\alpha}$ for $\alpha\in\Delta$. We shall not need the
exact values of the structure constants of $\g$ in this basis.

The space $V$ of imaginary octonions in $\Oct$ is 7-dimensional, is preserved by $G$ and is a minimal non-trivial representation of $G$. We say that a subspace $W$ is an \emph{null subalgebra} in $V$, if the product of any two elements in $W$ is equal to $0$. It is easy to see that all one-dimensional null subalgebras are exactly the subalgebras of the form $\langle c\rangle$, where $c$ is an isotropic vector in $V$ with respect to the symmetric form $B$ (or, equivalently, $i_c\Omega\wedge\i_c\Omega\wedge \Omega = 0$). It is known that $G$ acts transitively on the set of all such vectors.

Fix an isotropic vector $c\in V$ and consider its annihilator $A(c) = \{x\in V \mid xc = 0\}$. It is known to be a $3$-dimensional subalgebra of $\Oct$, which lies in $V$ and is no longer null. In terms of the 3-form $\Omega$ it can be described as:
\[
A(c) = \ker i_c\Omega = \{ v\in V \mid \Omega(c,v,\cdot)=0 \}.
\]

In fact, we can always find an element $g_c\in G$ such that $g_c.c=c_3$, so that $g_c.A(c)$ will become the subalgebra with the basis $\langle c_3, c_7, c_8\rangle$ and the multiplication $c_7c_8=-c_8c_7 =c_3$, $c_3^2=c_7^2=c_8^2=0$. (Here we use the notation $\{c_1,\dots, c_8\}$ for the basis in $\Oct$ as given in~\cite{humph}.) Thus, we see that there are no 3-dimensional null subalgebras in $V$, and any isotropic element $c$ is contained in a one-parameter family of $2$-dimensional null subalgebras of the form
$g_c^{-1}.\langle c_3, \lambda_7 c_7 + \lambda_8 c_8\rangle$, $[\lambda_1:\lambda_2]\in\RP^1$.

It is easy to see that $G$ acts transitively on the set of all flags $V_1\subset V_2$, where $V_1$ and $V_2$ are both null subalgebras of dimension $1$ and $2$ respectively. Fix any of such flags. Than its stabilizer under the action of $G$ is exactly the Borel subgroup of $G$, while the stabilizers of $V_1$ and $V_2$ are two different parabolic subgroups $P_1$ and $P_2$ of $G$.

Both homogeneous spaces $M_i=G/P_i$, $i=1,2$ are 5-dimensional, while $G/B$ is 6-dimensional. Let us describe the distribution $D$ and the notion of integral curves of constant type in each of these cases.

Let us start with the homogeneous space $G/B$. First, note that any flag of null subalgebras $V_1\subset V_2$ is naturally extended to the complete flag in $V$:
\[
0\subset V_1\subset V_2 \subset V_3 = A(V_1) \subset V_4 = V_3^{\perp} \subset V_5 = V_2^{\perp}\subset V_6 = V_1^{\perp} \subset V.
\]
As $\g_{-1} = \langle X_{-\alpha_1}, X_{-\alpha_2} \rangle$, we see that $\dim D=2$, and $D = \ker d\pi_1 \oplus \ker  d\pi_2$, where $\pi_i \colon G/B \to G/P_i$, $i=1,2$ are natural projections. This implies that a curve $\gamma$ of flags of null subalgebras:
\begin{equation}\label{curveB}
0 \subset V_1(t)\subset V_2(t) \subset V
\end{equation}
 is an integrable curve of $D$ if and only if it satisfies the following conditions:
\begin{equation}\label{intcurve}
V_1' \subset V_2,\quad V_2' \subset A(V_1).
\end{equation}
 It is easy to see that the action of $G$ on $PD$ has exactly 3 orbits: two are the kernels of $d\pi_i$, $i=1,2$ and one open orbit is complementary to the first two. So, all integral curves of constant type consist of the following:
\begin{enumerate}
\item fibers of the projection $\pi_1\colon G/B\to G/P_1$, or, in other words, curves~\eqref{curveB} with constant $V_1(t)$;
\item  fibers of the projection $\pi_2\colon G/B\to G/P_2$, or, in other words, curves~\eqref{curveB} with constant $V_2(t)$;
\item \emph{non-degenerate} curves~\eqref{curveB}, satisfying $V_1'=V_2$ and $V_2'=A(V_1)$.
\end{enumerate}
Elementary calculations show that a non-degenerate curve additionally satisfies: $V_1'''=V_3'=V_3^{\perp}$, $V_1^{(4)}=V_2^\perp$, $V_1^{(5)}=V_1^\perp$ and $V_1^{(6)}=V$. Thus, the complete osculating flag of the curve $V_1(t)$, considered as a projective curve in $PV$, is completely determined by $V_1$ and $V_2(t)=V_1'$ via algebraic operations only.

Let us now describe how the developed theory applies to the non-degenerate integral curves in $G/B$. In this case we can take $x=X_{-\alpha_1}+X_{-\alpha_2}$. Simple computation shows that
\begin{itemize}
\item the subalgebra $\sg$ is isomorphic to $\sll(2,\R)$ with the basis $x$, $h=H_1+H_2$ and $y=X_{\alpha_1}+X_{\alpha_2}$;
\item the space $(\sg+[x,\pg])\cap \pg$ coincides with $\sum_{i=1}^4 \g_i$;
\item the subspace $W=\g_5$ is $\sg^{(0)}$-invariant and is complementary to $\sg+[x,\pg]$.
\end{itemize}
Hence, applying Theorem~\ref{thm1} in this case, we see that we can canonically associate a projective connection and a W-valued fundamental invariant to each non-degenerate integral curve $\gamma$ on $G/B$. The projective connection defines a canonical projective parametrization $\tau$ on the curve $\gamma$, while fundamental invariant can be written as $f(\tau)(d\tau)^6$ and completely determines the equivalence class of $\gamma$.

Next, consider curves $\gamma$ in the homogeneous space $G/P_1$ of all $1$-dimensional null subalgebras in $V$. We shall write $\gamma=V_1(t)$. The distribution $D$ is 2-dimensional in this case and can be written as $D_{V_1} = A(V_1)/V_1$ for each point $V_1\in G/P_1$. Thus, we see that a curve $V_1(t)$ is integrable if and only if it satisfies $V_1'\subset A(V_1)$. We assume that $V_1(t)$  does not degenerate to a point. Then $V_2(t) = V_1'(t)$ is a 2-dimensional subspace in $A(V_1(t))$ for each $t\in\gamma$. As it also contains $V_1(t)$, we see that $V_2(t)$ is necessarily an null subalgebra. Simple computations show that $V_2'\subset A(V_1)$, and, hence, the curve $\gamma\subset G/P_1$ is naturally lifted to $G/B$. If, moreover, $V_2'=A(V_1)$, then the lifted curve is non-degenerate, and we can apply the above constructions to equip $\gamma$ with a canonical projective invariant and a single fundamental invariant.

Finally, consider curves $\gamma$ in the homogeneous space $G/P_2$ of all $2$-dimensional null subalgebras in $V$. In this case $\g_0\equiv\gl(2,\R)$, and the $\g_0$-module $\g_{-1}$ is equivalent to $S^3(\R^2)$. In particular, $\dim D = 4$, and for each $V_2\in G/P_2$ the space $D_{V_2}$ can be identified with a 4-dimensional subspace in $\Hom(V_2, V_2^{\perp}/V_2)$.  A curve $V_2(t)$ in $G/P_2$ is integral if and only if $V_2'\subset V_2^\perp$.

The action of $G$ on $PD$ has 4 different orbits corresponding to the orbits of $GL(2,\R)$ action on $S^3(\R^2)$, that is on the space of cubic polynomials:
\begin{enumerate}
\item polynomials with triple root;
\item polynomials with one simple and one double root;
\item polynomials with one real and 2 complex roots;
\item polynomials with 3 different real roots.
\end{enumerate}
We shall call integral curves, whose tangent lines belong to the first orbit \emph{special} curves. Let us describe them in more detail. Namely, for each null subalgebra $V_2\subset V$ we define the embedding of $PV_2$ into $PD_{V_2}$ as follows:
\[
PV_2\to PD_{V_2},\quad V_2 \supset V_1 \mapsto (V_2/V_1)^* \otimes A(V_1)/V_2.
\]

This embedding defines a field of rational normal curves in $PD$. This cone is exactly the orbit~(1) of the action of $G$ on $PD$. In particular, we see that a curve $\gamma=V_2(t)$ is special, if and only if $V_2'$ is 3-dimensional. In this case it is equal to $A(V_1(t))$, where $V_1 \subset V_2(t)$ is defined uniquely by the condition $V_1'\subset V_2$. Thus, as in the case of curves in $G/P_1$, any special curve in $G/P_2$ is naturally lifted to the curve in $G/B$. If, in addition, we suppose that $V_1'=V_2$, then this lift is a non-degenerate curve, and we can naturally equip $\gamma$ with canonical projective parameter and a single fundamental differential invariant.

Consider now the case, when $\gamma\subset G/P_2$ is an integral curve of constant type, and its tangent lines belong to the orbit~(2) above. Then we can no longer lift this curve to an integral curve on $G/B$ and have to deal with it on the homogeneous space $G/P_2$ itself. We can take the corresponding element $x\in\g_{-1}$ as $X_{-\alpha_1-\alpha_2}$. Direct computation shows that:
\begin{itemize}
\item the subalgebra $\sg$ is isomorphic to $\sll(2,\R)+\st(2,\R)$ with the basis $x$, $H_1, H_2$, $X_{\alpha_1+\alpha_2}$ and $X_{\alpha_1+3\alpha_2}$ (here by $\st(2,\R)$ we denote the 2-dimensional subalgebra of $\sll(2,\R)$ consisting of upper-triangular matrices);
\item the space $(\sg+[x,\pg])\cap \pg$ has codimension 2 in $\pg$ and is concentrated in degrees $0$ and $1$;
\item the subspace $W=\langle X_{\alpha_1}, X_{2\alpha_1+3\alpha_2}\rangle$ is $\sg^{(0)}$-invariant and is complementary to $\sg+[x,\pg]$.
\end{itemize}
Thus, we see that in this case we get a natural projective connection and a $W$-valued fundamental invariant on $\gamma$. Note that the $\sg^{(0)}$-invariant subspace $W$ exists in this case, even though the subalgebra $\sg$ is not reductive in $\g$.

Assume now that $\gamma\subset G/P_2$ is an integral curve of constant type, and its tangent lines belong to the orbit~(3) above. We can take the corresponding element $x\in\g_{-1}$ as $X_{-\alpha_1}+X_{-\alpha_1-3\alpha_2}$. Direct computation shows that:
\begin{itemize}
\item the subalgebra $\sg$ is isomorphic to $\sll(2,\R)$;
\item the space $(\sg+[x,\pg])\cap \pg$ has codimension 3 in $\pg$ and is concentrated in degrees $0$ and $1$;
\item the subspace $W=\langle X_{2\alpha_1+3\alpha_2}, X_{\alpha_1+\alpha_2}, X_{\alpha_1+2 \alpha_2}\rangle$ is $\sg^{(0)}$-invariant and is complementary to $\sg+[x,\pg]$.
\end{itemize}
Thus, we see that we can again associate a natural projective connection on $\gamma$ and 3 scalar fundamental invariants on top of it.

The case of the orbit of type~(4) can be considered in the same way and leads to the same results, as in case of orbit~(3).

\section{Conformal, symplectic and $G_2$-structures on the solution space of an ODE}
\label{sec:app}

\subsection{Compatibility of projective curves with symmetric and skew-symmetric bilinear forms}
As one of the applications of the theory we find the explicit conditions when the solution space of an ODE of order $k\ge 3$ carries a natural symplectic, conformal, or a $G_2$-structure.

As in Subsection~\ref{ssec:41} consider a non-degenerate curve $\gamma$ in $\RP^k$. Its dual curve $\gamma^* \subset \RP^{k,*}$ is defined as a curve of osculating hyperplanes $\gamma^{(k-2)}$. The curve $\gamma$ is called self-dual, if there is a projective transformation from $\RP^k$ to $\RP^{k,*}$ that maps $\gamma$ to $\gamma^*$. Wilczynski~\cite{wilch} proved that a curve is self-dual if and only if its Wilczynski invariants $\theta_{2i+1}$, $i=1,\dots,[k/2]$ vanish identically.

Let as give a simple proof of this result using the developed techniques as well as establish when a projective curve $\gamma$ in $\RP^6$ defines a unique $G_2$ subalgebra in $\gl(7,\R)$.

Let $V$ be an arbitrary finite-dimensional vector space of dimension $k+1$ and let $P(V)$ denote its projectivization. We say that a non-degenerate curve $\gamma\in P(V)$ is compatible with a non-degenerate bilinear form $b\colon V\times V\to \R$ if:
\begin{equation}\label{b}
b(\gamma^{(i)},\gamma^{(k-2-i)}) = 0\quad\text{for all }i=0,\dots,k-2.
\end{equation}

Let us show that the existence of such $b$ is equivalent to the above definition of self-duality. Indeed, if a curve $\gamma$ is compatible with a bilinear form $b$, then this form $b$ considered as a map from $V$ to $V^*$ defines a projective transformation that maps $\gamma$ to its dual curve $\gamma^{(k-2)}$.

Similarly, any projective transformation from $P(V)$ to $P(V^*)$ mapping $\gamma$ to $\gamma^{(k-2)}$ defines a bilinear from on $V$ such that $b(\gamma,\gamma^{(k-2)})=0$. This also implies that $b(\gamma,\gamma^{(k-3)})=0$. Differentiating this equality, we immediately get $b(\gamma',\gamma^{(k-3)})=0$. Proceeding in the same way we prove that the curve $\gamma$ satisfies~\eqref{b} .

\begin{lem}\label{lem:dual}
A non-degenerate curve $\gamma\subset P(V)$ is compatible with a non-degenerate symmetric ($\dim V$ is odd) or skew-symmetric form ($\dim V$ is even) if and only if all Wilczynski invariants $\theta_{2i+1}$, $i=1,\dots, [(\dim V-1)/2]$ of $\gamma$ vanish identically.
\end{lem}
\begin{proof}
The condition~\eqref{b} can be interpreted as saying that the flag
\[
0\subset\gamma\subset \gamma' \subset \dots \subset \gamma^{([\frac{k-1}{2}])}
\]
is isotropic with respect to $b$ and the spaces $\gamma^{(i)}$ with $i>[\frac{k-1}{2}]$ are exactly the orthogonal complements to $\gamma^{k-2-i}$ with respect to $b$.

Denote by $SL(V,b)$ the subgroup of $SL(V)$ preserving the form $b$. That is $SL(V,b)$ is either the group of orthogonal or symplectic matrices depending on the symmetry index of $b$. It is well-known that the group $SL(V,b)$ acts transitively on flags of isotropic subspaces. Thus, if the curve is compatible with a bilinear form $b$, then we can consider it as a curve in the flag variety of the group $SL(V,b)$. Denote by $\sll(V,b)$ the subalgebra of $\sll(V)$ corresponding to $SL(V,b)$.

Let $P\subset SL(V,b)$ be the canonical moving frame for this curve corresponding to the subspace $\widetilde W\subset \sll(V,b)$ defined by~\eqref{eq:w}. Note that the Killing form of $\sll(V,b)$ coincides up to a constant with the restriction of the Killing form of $\sll(V)$. Therefore, the same formula~\eqref{eq:w} in case of $\sll(V)$ defines the subspace $W$ containing $\widetilde W$. Hence, the moving frame $P$ is at the same time the canonical frame of $\gamma$ viewed as a curve in the corresponding (larger) flag variety of the group $SL(V)$.

Recall (see Subsection~\ref{ssec:41}) that the space $W$ is spanned by $y^i$, $i=2,\dots,k$, where $\{x,h,y\}$ is the standard basis of $\sll(2,\R)$ embedded irreducibly into $\sll(V)$. Decomposing $\sll(2,\R)$-module $\sll(V,b)$ into the sum of irreducible representations, we see that $y^i$ lies in $\sll(V,b)$ if and only if $i$ is odd. Then equation~\eqref{eq:omw} implies that the canonical moving frame for the curve $\gamma$ takes values in $\sll(V,b)$ if and only if all Wilczynski invariants $\theta_{2i+1}$ vanish identically.
\end{proof}

\subsection{Compatibility of projective curves with $G_2$-structures}
As in Subsection~\ref{g2-ex}, let $V\subset \Oct$ be the 7-dimensional vector space of imaginary split octonions, $G_2$ a subgroup of $GL(V)$ and let $\Omega\in \wedge^3 V^*$ be the $G_2$-invariant 3-form on $V$ that corresponds to the imaginary part of multiplication of imaginary octonions. In fact, $G_2$ coincides with the stabilizer of $\Omega$ under the natural action of $GL(V)$ on $\wedge^3 V^*$.

It is well-known that the action of $GL(7,\R)$ on the space of all $3$-forms has 2 open orbits. One of these orbits has a compact form of $G_2$ as a stabilizer (it corresponds to the multiplication tensor of the imaginary octonions) and another has a split real form of $G_2$ as a stabilizer (it corresponds to the multiplication tensor of split octonions). 

We shall call a 3-form on $\R^7$ \emph{a $G_2$-structure} if it lies in the orbit of $\Omega$ under the standard action of $GL(V)$ on $\wedge^3 V^*$, or, what is the same, if its stabilizer is conjugate to the split real form of $G_2$  in $GL(V)$. Abusing notation, we shall denote $G_2$-structures by the same letter $\Omega$.

We say that a non-degenerate curve $\gamma\subset P^6=P(V)$ is compatible with a $G_2$-structure $\Omega$, if
\[
\Omega(\gamma, \gamma'', \cdot) = 0.
\]
In terms of Subsection~\ref{g2-ex} this means that $\gamma$ is a non-degenerate integral curve on the generalized flag variety $G_2/B$, where $B$ is a Borel subgroup of $G_2$.

Using the same argument as in Lemma~\ref{lem:dual}, we easily prove that the curve $\gamma$ is compatible with a 3-form $\Omega$, if and only if the canonical frame of $\gamma$ lies in the stabilizer of $\Omega$. It is easy to see that the irreducible subalgebra $\sll(2,\R)\subset \sll(7,\R)$ is contained in a unique split form of the Lie algebra $G_2$. Moreover, this real form of $G_2$ viewed as an $\sll(2,\R)$-module is decomposed into the 3-dimensional and 11-dimensional irreducible submodules. The first one is $\sll(2,\R)$ itself, and the 11-dimensional submodule has $y^5$ as its highest weight vector. (Here, as above, we denote by $\{x,h,y\}$ the standard basis of $\sll(2,\R)$ viewed as an irreducible subalgebra of $\sll(7,\R)$.) Thus, similar to Lemma~\ref{lem:dual} we get:

\begin{lem}
A non-degenerate curve $\gamma\subset P^6$ is compatible with a $G_2$-structure if and only if all its Wilczynski invariants except for $\theta_6$ vanish identically.
\end{lem}

\subsection{Geometric linearization of ordinary differential equations}
Classical linearization procedure associates a linear differential equation with each solution of
a non-linear equation $\mathcal E$. This equation, called the linearization along a given solution, describes all possible first order deformations of the solution that still satisfy the original differential equation up to the first order of deformation parameter. In more detail, let
\begin{equation}\label{ode}
y^{(k+1)}=f(x,y,y',\dots,y^{(k)})
\end{equation}
be an arbitrary ordinary differential equation~$\E$ of order $(k+1)$ solved with respect to the highest derivative. Suppose $\bar y(x)$ is an arbitrary solution of~\eqref{ode}. Consider its deformation $y(x) = \bar y(x) + \eps z(x)$, substitute it to~\eqref{ode} and expand it with respect to $\eps$. As $\bar y$ is a solution, the zero-order term will vanish identically. The first order term will define the linear equation on $z(x)$:
\begin{equation}\label{linode}
z^{(k+1)} = \frac{\partial f}{\partial y} z + \frac{\partial f}{\partial y'} z' + \dots + \frac{\partial f}{\partial y^{(k)}} z^{(z)},
\end{equation}
where all partial derivatives  $\partial f/\partial y^{(i)}$, $i=0,\dots,k$, are evaluated at the solution $\bar y(x)$. Equation~\eqref{linode} is called \emph{the linearization} of the equation~\eqref{ode} along the solution~$\bar y$. The solution space of the linearization~\eqref{linode} can be considered as the tangent space to the solution space $\operatorname{sol}(\E)$ of the original equation~$\E$ given by~\eqref{ode}.

As it was originally discovered by Wilczynski~\cite{wilch}, the geometry of linear differential equations is equivalent to the projective geometry of curves. In the modern language this equivalence can be formalized as follows. Let $M\to \Gr_k(V)$ be an embedding of a smooth manifold $M$ into the Grassmann variety $\Gr_k(V)$. Let $\pi\colon \mathbb E\to \Gr_k(V)$ be the hyperplane vector bundle over $\Gr_k(V)$, which is dual to the canonical one, i.e. $\mathbb E_W=W^*$ for each $W\in \Gr_k(V)$. Each element $\alpha\in V^*$ naturally defines the global section $s_\alpha$ of $\mathbb E$. Let $E$ be the restriction of $\mathbb E$ on $M$, and let $S\subset\Gamma(E)$ be the restriction of all sections $s_\alpha$ to $M$. Then $S$ is a well-defined finite-dimensional subspace in the space of sections of~$E$. Prolonging elements of $S$ to the jet spaces $J^i(E)$, we can construct a differential equation $\mathcal E\subset J^r(E)$ whose solution space coincides with~$S$.

In other direction, let $\pi\colon E\to M$ be an arbitrary $k$-dimensional vector bundle over the manifold $M$ and let $\mathcal E\subset J^r(E)$ be a finite-type linear differential equation on sections of $E$. Let $S=\operatorname{sol}(\mathcal E)$ be the solution space of $\mathcal E$. Let $V=S^*$ and define the embedding $M\to \Gr_k(V)$ as follows:
\[ x\mapsto \{s\in S \mid s(x) = 0\}^{\perp} \subset V.\]
This correspondence between finite-type linear differential equations and embedded submanifolds in Grassmann varieties is extensively used in work of Se-ashi~\cite{se-ashi} for describing the invariants of linear systems of ordinary differential equations.

\subsection{Geometric structures on the solution space of an ODE}


Let $M=\sol(\E)$ be the solution space of the equation~\eqref{ode}. For any solution $\bar y\in M$ we get the curve $\gamma\colon \R \to PT_{\bar y}(M)$ defined by the linearization~\eqref{linode} of the equation~\eqref{ode} at the solution $\bar y$. Wilczynski invariants $\theta_i$, $i=3,\dots,k+1$ for each such linearization define the so-called \emph{generalized Wilczynski invariants} $\Theta_i$, $=3,\dots,k+1$ of the initial non-linear ODE. See~\cite{dou:08} for more details.

Recall that a \emph{pseudo-conformal structure} on a smooth manifold $M$ is defined as a family of pseudo-Riemannian metrics on $M$ proportional to each other by non-zero functional factor. Similarly, we define the \emph{pseudo-symplectic structure} on $M$ as a family of non-degenerate 2-forms on $M$ proportional to each other by non-zero functional factor. Finally, \emph{a $G_2$-structure} on $M$ is defined as a differential 3-form $\Omega$ such that $\Omega_p$ defines a $G_2$-structure on each tangent space $T_pM$.

We say that a pseudo-symplectic, pseudo-conformal or $G_2$-structure on the solution space $M$ of the given ODE is \emph{natural}, if it is compatible with linearizations $\gamma \to PT_\gamma(M)$ for each solution $\gamma$. Using the above results on the compatibility of projective curves with symmetric and skew-symmetric bilinear forms and $G_2$-structures, we immediately get the following result.

\begin{thm}\label{thm:2}
Let $y^{(k+1)}=f(x,y,y',\dots,y^{(k)})$ be an ordinary differential equation of order $k+1\ge 3$. Its solution space is admits a natural:
\begin{enumerate}
\item pseudo-conformal structure if and only if $k$ is odd and all generalized Wilczynski invariants $\Theta_{2i+1}$ vanish identically;
\item pseudo-symplectic structure if and only if $k$ is even and all generalized Wilczynski invariants $\Theta_{2i+1}$ vanish identically;
\item $G_2$-structure if and only if $\Theta_3=\Theta_4=\Theta_5=\Theta_7=0$.
\end{enumerate}
In all these cases such structure is uniquely defined.
\end{thm}

\section{Generalizations and discussion}\label{sec:5}
\subsection{Parametrized curves}
The paper is easily generalized to the case of parametrized curves. In this case one needs to change the definition of the subalgebra $\sg$ as follows:
\begin{align*}
\sg_{-i} & = 0,\quad i \ge 0;\\
\sg_{-1} & = \langle x \rangle;\\
\sg_0  &= \{ u \in\g_0 \mid [x,u] = 0 \};\\
\sg_{i+1} &= \{ u \in \g_{i+1} \mid [x,u]\subset \sg_i \}.
\end{align*}
We also have to consider tangent spaces to a parametrized curve $\gamma$ as elements of $\g/\hg$ instead of one-dimensional subspaces. As a result, we get a result similar to Theorem~\ref{thm1}, which is applicable to all parametrized curves $\gamma$ such that their images (treated as unparametrized curves) can be approximated by $\gamma_0$ up to the first order.

\subsection{Curves in general parabolic geometries}
By a parabolic geometry modeled by $G/P$ we understand the principal $P$-bundle $\pi\colon\cG\to M$ equipped with a 1-form $\omega\colon T\cG\to \g$ such that:
\begin{enumerate}
\item $\omega_p$ is an isomorphism of vector spaces for any $p\in \cG$;
\item $\omega (X^*) = X$ for any fundamental vector field $X^*$ on $\cG$ corresponding to an element $X\in \pg$;
\item $R_g^*\omega = \Ad g^{-1}\omega$ for any right shift $R_g$ defined by the action of an element $g\in P$.
\end{enumerate}
The classical examples of parabolic geometries include the projective and conformal structures on smooth manifolds. The tangent space $TM$ is naturally identified with $\cG\times_P (\g/\pg)$ and, as in the case of a homogeneous space $G/P$, carries a natural distribution $D\subset TM$ corresponding to the subspace $(\g_{-1}+\pg)/\pg$. In particular, we can also introduce the notion of an integral curve $\gamma$ in $M$. Using again the identification $TM\equiv \cG\times_P(\g/\pg)$ we can define the notion of a curve of constant type $x$ for any $x\in \g_{-1}$.

The notion of a moving frame along $\gamma$ can be naturally generalized to any parabolic geometry as an arbitrary subbundle of a $P$-bundle $\pi^{-1}(\gamma)\to\gamma$.

Theorem~\ref{thm1} is also immediately generalized as follows:
\begin{thm}\label{thm1b}
Let $(M,\cG,\om)$ be a parabolic Cartan geometry modeled by a generalized flag variety $G/P$ and let $x$ be an arbitrary non-zero element in $\g_{-1}$. Fix any graded subspace $W\subset\pg$ complementary to $\sg^{(0)} + [x,\pg]$.

Then for any integral curve $\gamma$ in $M$ of type $x$ there is a canonical frame bundle $E\subset \cG$ along $\gamma$ such that for each point $p\in E$ the subspace $\om(T_pE)$ is $W$-normal.

If, moreover, $W$ is $S^{(0)}$-invariant, then $E$ is a principal $S^{(0)}$-bundle and $\om|_{E}$ decomposes as a sum of the $W$-valued 1-form $\om_W$ and the $\sg$-valued 1-form $\om_{\sg}$, where $\om_\sg$ is a flat Cartan connection on $\gamma$ defining a canonical projective parametrization of $\gamma$, and $\om_W$ is a vertical $S^{(0)}$-equivariant 1-form.
\end{thm}

\subsection{Submanifolds of higher dimension in generalized flag varieties}
We briefly outline how the results of this paper are generalized to the submanifolds of arbitrary dimension in generalized flag varieties (but not in curved parabolic geometries).

For this purpose the element $x\in\g_{-1}$ needs to be replaced by an arbitrary commutative subalgebra $\xg\subset \g_{-1}$. We say that a submanifold $X\subset G/P$ has \emph{constant type $\xg$}, if for any point $p\in X$ there exists an element $g\in G$, such that $g.o=p$ and $g_*(\xg)=T_pX$. In other words, $X$ is of type $\xg$, if all tangent spaces $T_pX$, $p\in X$ belong to the orbit of the subspace $\xg\in T_oM$ under the natural action of $G$ on $\Gr_r(TM)$, where $r=\dim \xg$.

As in case of unparametrized curves, we can define a graded subalgebra $\sg\subset\g$ by
\begin{align*}
\sg_{-1} &= \xg;\\
\sg_{i} &= \{ u \in \g_{i} \mid [u,\xg]\subset \sg_{i-1} \}\quad\text{for all }i\ge 0.
\end{align*}
Then according to~\cite{dk} $\sg$ is exactly the symmetry algebra of the orbit of the commutative subgroup $\exp(\xg)\subset G$ through the origin $o=eP$. We call this orbit the \emph{flat} submanifold of type $\xg$.

Denote by $C^k(\sg_{-1},\g/\sg)$ the standard cochain complex corresponding to the $\sg$-module $\g/\sg$.
The grading of $\g$ naturally extends to the gradings of all spaces $C^k(\sg_{-1},\g/\sg)$, as well as their cochain and coboundary subspaces $Z^k(\sg_{-1},\g/\sg)$ and $B^k(\sg_{-1},\g/\sg)$. Denote by $Z_{+}^k(\sg_{-1},\g/\sg)$ and $B_{+}^k(\sg_{-1},\g/\sg)$ their subspaces spanned by elements of positive degree.

Let $W$ be an arbitrary graded subspace in $Z^1_+(\sg_{-1},\g/\sg)$ complementary to $B^1_+(\sg_{-1},\g/\sg)$. Generalizing the above notion of $W$-normality, we say that a subspace $U\subset \g$ is $W$-normal (with respect to $\sg$), if
\begin{enumerate}
\item $\gr U = \sg$;
\item there is a map $\chi\colon \sg_{-1}\to \g$ such that the corresponding quotient map $\bar \chi\colon \sg_{-1}\to \g/\sg$ lies in $W$ and $x+\chi(x)\in U$ for all $x\in\sg_{-1}$.
\end{enumerate}
As in the case $\dim \sg_{-1}=1$, it is easy to check that the element $\bar\chi\in W$ is uniquely defined for any $W$-normal subspace $U\in \g$.

Then we can prove that for any submanifold $X$ of type $\xg=\sg_{-1}$ there is a unique moving frame $E\subset G$ along $X$ such that $\omega(E_p)$ is a $W$-normal subspace for any point $p\in X$.
In particular, we see that if $H^1_+(\sg_{-1},\g/\sg)=0$, then any submanifold of type $\sg_{-1}$ is locally
equivalent to the flat one.

\end{document}